\documentclass[oneside,english]{amsart}
\usepackage[T1]{fontenc}
\usepackage[latin9]{inputenc}
\usepackage{verbatim}
\usepackage{amstext}
\usepackage{amsthm}
\usepackage{amssymb}

\makeatletter
\numberwithin{equation}{section}
\numberwithin{figure}{section}
\theoremstyle{plain}
\newtheorem{thm}{\protect\theoremname}[section]
  \theoremstyle{definition}
  \newtheorem{defn}[thm]{\protect\definitionname}
  \theoremstyle{remark}
  \newtheorem{rem}[thm]{\protect\remarkname}
  \theoremstyle{remark}
  \newtheorem*{rem*}{\protect\remarkname}
  \theoremstyle{plain}
  \newtheorem{lem}[thm]{\protect\lemmaname}

\usepackage{tikz}
\usetikzlibrary{shapes,arrows}
\usepackage{tikz-qtree}

\usepackage{titlesec}

\titleformat{\section}
  {\normalfont\Large\bfseries}{\S\thesection}{1em}{}

\makeatother

\usepackage{babel}
  \providecommand{\definitionname}{Definition}
  \providecommand{\lemmaname}{Lemma}
  \providecommand{\remarkname}{Remark}
\providecommand{\theoremname}{Theorem}

\begin{document}
\global\long\def\sphi{\mathcal{S}_{\phi}}
\global\long\def\dens{\text{dens}}
\global\long\def\A{\text{\ensuremath{\mathbf{A}}}}
\global\long\def\x{\mathtt{x}}
\global\long\def\F{\mathtt{F}}
\global\long\def\Th{\mathrm{Th}}

\title{Thicket Density}

\author{Siddharth Bhaskar}

\address{DIKU, University of Copenhagen.}
\begin{abstract}
We define a new type of ``shatter function'' for set systems that
satisfies a Sauer-Shelah type dichotomy, but whose polynomial-growth
case is governed by Shelah's 2-rank instead of VC dimension. We identify
the rate of growth of this shatter function, the quantity analogous
to VC density, with Shelah's $\omega$-rank.
\end{abstract}

\thanks{This line of inquiry began during my graduate study when my advisor,
Yiannis Moschovakis, instructed me to find the essential model-theoretic
content in Tiuryn \cite{Tiu89}. I discovered the central outline
of the results after the fourth or fifth time reading that paper.
These have since been expanded and refined over several years, with
the help of Anton Bobkov, Caroline Terry, Alex Kruckman, Jim Freitag,
David Lippel, and Ross Berkowitz. Thanks as well to Larry Moss, Steve
Lindell, and Scott Weinstein for their support while at Indiana University
and Haverford College. Finally, thanks again to Larry Moss and Alex
Kruckman for helping proofread various iterations of this manuscript.}
\maketitle

\section{Introduction}

The \emph{shatter function }is a function of type $\mathbb{N\to N}$
that measures of the complexity of a set system. The shatter function
of any set system satisfies the \emph{Sauer-Shelah dichotomy: }it
is either the binary exponential function $n\mapsto2^{n}$, or is
polynomially bounded. Whether or not the shatter function is polynomially
bounded or exponential depends on whether a certain integer parameter,
the \emph{VC dimension}, is finite or infinite. In the finite case,
the growth rate of the shatter function is a real number called the
\emph{VC density}.

VC density was discovered by Vapnik and Chervonenkis \cite{VC71}
and found important applications in probability theory, combinatorial
geometry, and computational learning theory.\footnote{For an excellent exposition, see Assouad \cite{Ass83}.}
The relevance of VC density to theories without the independence property
was pointed out by Laskowski \cite{Las90}, and subsequently developed
by Aschenbrenner et al. \cite{A15}.

In the present paper, we associate a new function of type $\mathbb{N\to N}$
with any set system, which we call the \emph{thicket shatter function}.
It also satisfies the Sauer-Shelah dichotomy, but the quantity that
distinguishes between polynomial and exponential growth is an instance
of Shelah's local 2-rank, and its rate of growth is an instance of
Shelah's local $\omega$-rank. In this context, we call these two
quantities \emph{thicket dimension }and \emph{thicket density }to
emphasize the analogy. 

Seen from another angle, our work can be read as a new way to calculate
Shelah's local $\omega$-rank using the asymptotic growth of certain
finite combinatorial objects. Notably, this can be performed in \emph{any
}model of a theory, not just a saturated one.

Our work was foreshadowed by Tiuryn, whose Lemma 3.6 in \cite{Tiu89}
contains a special case of our Theorem \ref{thm:dichotomy} below,
the Sauer-Shelah dichotomy for thicket shatter functions. It is remarkable
that he was concerned with problems in dynamic logic\textemdash at
best a distant relative of model theory, and even further from geometry
and computational learning theory.

\subsection*{Organization of this paper}

In Section \ref{sec:Thicket-dimension}, we introduce thicket dimension,
shatter function, and density, and prove the thicket version of the
Sauer-Shelah dichotomy. In Section \ref{sec:Ranks}, we identify thicket
density with a local model-theoretic rank, and in Section \ref{sec:Degree},
we formulate a notion of degree, or multiplicity, for thicket density.

\section{Trees\label{sec:Thicket-dimension}}

Our fundamental objects of study are set systems and binary trees.
The latter comes in two varieties: \emph{ordered}, and \emph{unordered},
and this refers to whether we distinguish left from right children
of non-leaves. Trees are normatively ordered; we shall say ``unordered
trees,'' when we mean it.
\begin{defn}
A \emph{tree} is either a single leaf or an ordered pair of subtrees,
which we call \emph{left }and \emph{right}. An \emph{unordered tree
}is either a single leaf or an unordered pair of subtrees.
\end{defn}

Notice that this definition allows for both finite trees and infinite
trees of depth $\omega$. In a set-theoretic account, a tree $T$
would be defined as a prefix-closed subset of $2^{<\omega}$ such
that for every $u\in2^{<\omega}$, $u0\in T\iff u1\in T$. We prefer
the ``coinductive data type'' definition presented here, which has
the advantage of giving a more succinct definition of unordered trees.
However, we will freely imagine a tree as a set of vertices, some
of which are \emph{leaves}, equipped with a partial order for the
ancestral relation. We trust that this will cause no difficulty.
\begin{defn}
\label{def:treeorder}For a tree $T$, and vertices $u,v,w\in T$,
we say

\begin{itemize}
\item $u\prec v$ in case $v\neq u$, but $v$ is contained in the subtree
with root $u$,
\item $u\prec_{L}v$ if $v$ is contained in the subtree whose root is the
left child of $u$, and
\item $u\prec_{R}v$ if $v$ is contained in the subtree whose root is the
right child of $u$.
\end{itemize}
For fixed $v\in T$, the set of vertices $P(v)=\{u:u\prec v\}$ is
linearly ordered by $\prec$ and is partitioned by $P_{L}(v)=\{u:u\prec_{L}v\}$
and $P_{R}(v)=\{u:u\prec_{R}v\}$. For example in Figure 2.1, $P_{L}(\ell_{3})=\{u\}$
and $P_{R}(\ell_{3})=\{v,w\}$.

\begin{figure}
\begin{tikzpicture}
\Tree
[ . $u$
	[ . $v$
		[ . $\ell_1$
		]
		[ . $w$
			[ . $\ell_2$
			]
			[ . $\ell_3$
			]
		]
	]
	[ . $\ell_4$
	]
]
\end{tikzpicture}\caption{A binary tree with leaves $\ell_{1},\ell_{2},\ell_{3},\ell_{4}$ and
non-leaves $u,v,w$. Depending on the context, this could be ordered
or unordered.}

\end{figure}
\end{defn}

For vertices $u,v$ in an unordered binary tree $T$, we cannot of
course say that $u\prec_{L}v$ or $u\prec_{R}v$. What we can say
is that $v$ and $w$ are contained in the same\emph{ }subtree of
$u$, or different\emph{ }subtrees of $u$.
\begin{defn}
For an unordered tree $T$, and vertices $u,v,w\in T$, we say
\begin{itemize}
\item $u\prec v$ in case $v\neq u$, but $v$ is contained in the subtree
with root $u$,
\item $v\sim_{u}w$ in case $u\prec v$, $u\prec w$, and $v$ and $w$
lie in the same subtree of $u$, and
\item $v\perp_{u}w$ in case $u\prec v$, $u\prec w$, and $v$ and $w$
lie in different subtrees of $u$.
\end{itemize}
Notice that, for any $v$ and $w$, there is a unique $u$ such that
$v\perp_{u}w$. For example, if we interpret the tree in Figure 2.1
to be unordered, then $\ell_{2}\sim_{u}\ell_{3}$, $\ell_{2}\sim_{v}\ell_{3}$,
but $\ell_{2}\perp_{w}\ell_{3}$.
\end{defn}

An important example of trees are the \emph{finite, balanced }trees.
\begin{defn}
The tree $B_{0}$ is a single leaf. The tree $B_{n+1}$ is the ordered
pair of trees $(B_{n},B_{n})$. Similarly, we can define the unordered
tree $B_{n}^{\circ}$.
\end{defn}

\begin{defn}
\label{def:embedding}An \emph{embedding }of the tree $T_{1}$ into
the tree $T_{2}$ is an injection of the vertices of $T_{1}$ into
the vertices of $T_{2}$ that preserves the $\prec_{L}$ and $\prec_{R}$
relations. An \emph{embedding }of the unordered tree $T_{1}$ into
the unordered tree $T_{2}$ is an injection of the vertices of $T_{1}$
into the vertices of $T_{2}$ that preserves the $\prec$, $\sim$,
and $\perp$ relations. The \emph{dimension }$d$ of a tree (respectively,
unordered tree) $T$ is the largest $n$ such that $B_{n}$ (respectively,
$B_{n}^{\circ}$) can be embedded into $T$, or $\infty$ if there
are arbitrarily large such $n$.
\end{defn}

\begin{rem}
For finite (ordered or unordered) trees $T$, dimension satisfies
the following useful recursive identity. If $T$ is a leaf, then $d=0$.
Otherwise, if $d_{1}$ and $d_{2}$ are the dimensions of its two
subtrees, then
\[
d=\begin{cases}
\max\{d_{1},d_{2}\} & \text{if }d_{1}\neq d_{2}\\
d_{1}+1 & \text{if }d_{1}=d_{2}.
\end{cases}
\]
\end{rem}

\section{Labeled trees and their solutions}

A \emph{set system} $(X,\mathcal{F},\in)$ is a two-sorted structure,
with sorts $X$ and $\mathcal{F}$, and equipped with a single binary
relation $\in\subseteq X\times\mathcal{F}$. Usually, we shall just
write $(X,\mathcal{F})$, suppressing $\in$. A typical example of
a set system is obtained by taking $X$ to be some set, $\mathcal{F}$
to be a family of subsets of $X$, and $\in$ to be the containment
relation. (But, in general, elements in $\mathcal{F}$ need not be
extensional.) Given any set system $(X,\mathcal{F},\in)$, its \emph{dual
}is the set system $(\mathcal{F},X,\in^{\star})$, where $F\in^{\star}x\iff x\in F$.
Clearly, the dual of the dual of any set system is identical to the
original.
\begin{defn}
Let $(X,\mathcal{F})$ be a set system and $T$ be a tree. An \emph{$X$-labeling
}of $T$ is an assignment of elements of $X$ to non-leaves of $T$.
If $T$ is an $X$-labeled tree with labeling $u\mapsto x_{u}$, and
if $v\in T$ is a leaf, then we say $F\in\mathcal{F}$ \emph{solves}
$v$ in case $(\forall u\prec v)(x_{u}\in F\iff u\prec_{L}v),$or
equivalently, $F\cap P(v)=P_{L}(v).$ In the special case that $T$
has depth 0, i.e., is a single leaf $v$, the quantifier $\forall u\prec v$
is vacuous, so $v$ has a solution in $\mathcal{F}$ iff $\mathcal{F}$
is nonempty.

\begin{figure}
[h]
\begin{tikzpicture}
\Tree
[.$1$
	[.$1$
		[.$2$
			[.$\{1,2\}$
			]
			[.$\{1,3\}$
			]
		]
		[.$0$
			[.$\circ$
			]
			[.$\circ$
			]
		]
	]
	[.$2$
		[.$3$
			[.$\{2,3\}$
			]
			[.$\{2,0\}$
			]
		]
		[.$10$
			[.$\{3,10\}$
			]
			[.$\{3,0\}$
			]
		]
	]
]
\end{tikzpicture}

\caption{A $\mathbb{N}$-labeling of $B_{3}$. Leaves are labeled with solutions
in ${\mathbb{N} \choose 2}$, the family of all 2-element subsets
of $\mathbb{N}$, when they have one.}
\end{figure}
\end{defn}

If $T$ is an unordered tree, an $X$-labeling is an assignment of
vertices of $T$ to elements of $X$. The solution of a single leaf
is meaningless, but we can still speak of a solution to the whole
tree.
\begin{defn}
If $T$ is an $X$-labeled unordered tree and $L\subseteq T$ is the
set of its leaves, a \emph{solution }to $T$ (in $\mathcal{F}$) is
an assignment $L\to\mathcal{F}$, denoted $v\mapsto F_{v}$, satisfying
\[
v_{1}\sim_{u}v_{2}\implies x_{u}\in F_{v_{1}}\leftrightarrow x_{u}\in F_{v_{2}},
\]
\[
v_{1}\perp_{u}v_{2}\implies x_{u}\in F_{v_{1}}\not\leftrightarrow x_{u}\in F_{v_{2}}.
\]
\end{defn}

In other words, any pair of sets labeling two leaves must disagree
on whether they include the element of $X$ labeling their most recent
common ancestor, and agree on all other common ancestors.
\begin{defn}
\label{def:admitting/forbidding trees}Say that a set system $(X,\mathcal{F})$
\emph{admits }a tree $T$ in case there is an $X$-labeling of $T$
such that each leaf of $T$ has a solution in $\mathcal{F}$. Similarly,
$(X,\mathcal{F})$ \emph{admits }an unordered tree $T$ in case there
exists an $X$-labeling of $T$ which has a solution in $\mathcal{F}$.
If a set system does not admit a tree or an unordered tree, then it
\emph{forbids }that tree.
\end{defn}

\begin{rem}
For a typical tree $T$, it is a much stronger statement to say that
$(X,\mathcal{F})$ admits $T$, rather than $(X,\mathcal{F})$ admits
the underlying unordered tree of $T$. Similarly, it is much stronger
to say that $(X,\mathcal{F})$ forbids the underlying unordered tree
of $T$, rather than $(X,\mathcal{F})$ forbids $T$. However, when
$T=B_{n}$, there is no difference: admitting or forbidding $B_{n}$
is equivalent to admitting or forbidding $B_{n}^{\circ}$.
\end{rem}

\begin{rem}
\label{rem:partition the labeling}Suppose $T$ is an unordered tree
with subtrees $T_{1}$ and $T_{2}$. Then if $(X,\mathcal{F})$ admits
$T$, there is an $X$-labeling of $T$ with a solution in $\mathcal{F}$.
Any solution can be partitioned into the leaves labeling $T_{1}$
and the leaves labeling $T_{2}$. If $x$ is the label of the root
of $T$, then the two parts of this solution disagree on $x$. Let
$\mathcal{F}_{x}$ and $\mathcal{F}_{\bar{x}}$ be the elements of
$\mathcal{F}$ containing and excluding $x$ respectively. Then, either
$(X,\mathcal{F}_{x})$ admits $T_{1}$ and $(X,\mathcal{F}_{\bar{x}})$
admits $T_{2}$, or $(X,\mathcal{F}_{x})$ admits $T_{2}$ and $(X,\mathcal{F}_{\bar{x}})$
admits $T_{1}$.

Conversely, suppose that for some $x\in X$, $(X,\mathcal{F}_{x})$
admits $T_{1}$ and $(X,\mathcal{F}_{\bar{x}})$ admits $T_{2}$.
Then $(X,\mathcal{F})$ admits $T$, by combining the $X$-labelings
of $T_{1}$ and $T_{2}$ into an $X$-labeling of $T$, labeling the
root by $x$. Similarly, if $(X,\mathcal{F}_{x})$ admits $T_{2}$
and $(X,\mathcal{F}_{\bar{x}})$ admits $T_{1}$, then $(X,\mathcal{F})$
admits $T$.
\end{rem}

\begin{defn}
The \emph{thicket dimension} $\dim(X,\mathcal{F})$ is the largest
$n$ such that $B_{n}$ is admissible, or $\infty$ if there are arbitrarily
large such $n$. If there are no such $n$, equivalently if $\mathcal{F}$
is empty, the dimension is $-1$.
\end{defn}

\begin{rem*}
This is a well-known quantity that occurs in many different contexts.
The thicket dimension of $(X,\mathcal{F})$ is equal to Shelah's rank
$R(\x=\x,\{\varphi\},2)$, where $\varphi$ is the formula $\x\in\F$,
relative to the theory $\Th(X,\mathcal{F})$ \cite{Shelah,Hod97}.
It is also called \emph{Littlestone dimension }in the context of computational
learning theory \cite{CF19}. Hodges calls it the \emph{branching
index }\textbf{\emph{\cite{Hod97}.}}
\end{rem*}

\begin{defn}
For a given set system $(X,\mathcal{F})$, let $\rho(n)$ be the maximum,
as $T$ varies over $X$-labelings of $B_{n}$, of the number of leaves
of $T$ with solutions in $\mathcal{F}$. The resulting function $\rho:\mathbb{N\to N}$
is the \emph{thicket shatter function }associated with $(X,\mathcal{F})$. 
\end{defn}

\begin{rem}
The thicket shatter function of any set system is bounded above by
the binary exponential function $n\mapsto2^{n}$.
\end{rem}

\subsubsection*{Dual Quantities}

Given a set system $(X,\mathcal{F})$ and a tree $T$, an \emph{$\mathcal{F}$-labeling
}of $T$ is an assignment of vertices of $T$ to elements of $\mathcal{F}$.
If $T$ is an $X$-labeled tree with labeling $u\mapsto F_{u}$, and
if $v\in T$ is a leaf, then we say $x\in X$ is a \emph{solution
}to $v$ in case $(\forall u\prec v)(F_{u}\in x\iff u\prec_{L}v).$
This allows us to define corresponding ``dual'' versions of thicket
dimension and the thicket shatter function, which are identical, respectively,
to the thicket dimension and thicket shatter function of the dual
set system $(X,\mathcal{F})^{\star}$.

VC dimension and dual VC dimension are bound within a single exponent;
that is, each is bound by 2 raised to the power of the other. On the
other hand, thicket dimension and dual thicket dimension are bound
within a double exponent. We do not know whether this is known to
be tight.

\section{The Sauer-Shelah dichotomy}

We now come to our first central result, that the Sauer-Shelah Lemma,
relating the growth of the (usual) shatter function to VC dimension,
holds verbatim in the thicket context.
\begin{thm}
\label{thm:dichotomy}Let $\chi(n,-1)=0$ and, for $k<\omega$, $\chi(n,k)={n \choose 0}+{n \choose 1}+\dots+{n \choose k}$.
For any set system $(X,\mathcal{F})$ and $k\in\{-1\}\cup\omega$,
\[
\dim(X,\mathcal{F})=\infty\implies\forall n\ \rho(n)=2^{n}
\]
\[
\dim(X,\mathcal{F})\le k\implies\forall n\ \rho(n)\le\chi(n,k).
\]
\end{thm}

The first implication is immediate: if $\dim(X,\mathcal{F})=\infty$,
then $(X,\mathcal{F})$ admits $B_{n}$ for each $n$, so $\rho(n)=2^{n}$.
If however $\dim(X,\mathcal{F})\le k$, then $(X,\mathcal{F})$ forbids
$B_{k+1}$. The conclusion follows by the second sentence of the next
theorem.
\begin{thm}
\label{thm:fundamental thm forbidden trees}For every finite unordered
tree $T$ of dimension $d$, there's a function $f(n)\in O(n^{d-1})$
such that if any set system $(X,\mathcal{F})$ forbids $T$, then
$\rho(n)\le f(n)$. If, in particular, $T$ is $B_{d}$, then $\rho(n)\le\chi(n,d-1)$.
\end{thm}

\begin{proof}
Suppose that $(X,\mathcal{F})$ forbids $T$. We proceed by induction
on the construction of $T$. If $T$ has dimension 0, it is the single
leaf $B_{0}$. If $(X,\mathcal{F})$ forbids $T$, then $\mathcal{F}$
must be empty, so $\rho(n)=0$, which is $O(n^{-1})$.

Otherwise, suppose that $T$ has subtrees $T_{1}$ and $T_{2}$. Suppose
their dimensions are $d_{1}$ and $d_{2}$, and suppose that $f_{1}$
and $f_{2}$ are given by induction. For $x\in X$, let $\mathcal{F}_{x}$
be the collection of those sets in $\mathcal{F}$ that include $x$,
and let $\mathcal{F}_{\bar{x}}$ be the collection of those which
exclude $x$. Let $\rho_{x}(n)$ and $\rho_{\bar{x}}(n)$ be the thicket
shatter functions of $(X,\mathcal{F}_{x})$ and $(X,\mathcal{F}_{\bar{x}})$
respectively. It is easy to see that for any $x\in X$, $\rho(n)\le\rho_{x}(n)+\rho_{\bar{x}}(n);$
simply take the $X$-labeled tree witnessing $\rho(n)$ and observe
that the solutions to its leaves can be partitioned into those containing
$x$ and those not containing $x$.

Let $P_{i}(x)$ express that $(X,\mathcal{F}_{x})$ admits $T_{i}$,
and $Q_{i}(x)$ express that $(X,\mathcal{F}_{\bar{x}})$ admits $T_{i}$.
Then by Remark \ref{rem:partition the labeling},
\[
(X,\mathcal{F})\text{ admits }T\iff\exists x\in X\:\big((P_{1}(x)\wedge Q_{2}(x))\vee(P_{2}(x)\wedge Q_{1}(x))\big).
\]
Reasoning propositionally, $(X,\mathcal{F})$ forbids $T$ if and
only if for all $x\in X$,
\[
(\neg P_{1}(x)\wedge\neg P_{2}(x))\vee(\neg P_{1}(x)\wedge\neg Q_{1}(x))\vee(\neg Q_{2}(x)\wedge\neg P_{2}(x))\vee(\neg Q_{2}(x)\wedge\neg Q_{1}(x)).
\]
By induction, this implies, for all $x\in X$,
\[
(\rho_{x}\le f_{1}\wedge\rho_{x}\le f_{2})\vee(\rho_{x}\le f_{1}\wedge\rho_{\bar{x}}\le f_{1})\vee(\rho_{\bar{x}}\le f_{2}\wedge\rho_{x}\le f_{2})\vee(\rho_{\bar{x}}\le f_{2}\wedge\rho_{\bar{x}}\le f_{1}),
\]
where, e.g., $\rho_{x}\le f_{1}$ abbreviates $\forall n\in\mathbb{N}\,\rho_{x}(n)\le f_{1}(n)$.
Label the four disjuncts (i)-(iv).

Suppose that for some $x\in X$, case (ii) holds. Then $\rho(n)\le\rho_{x}(n)+\rho_{\bar{x}}(n)\le2f_{1}(n)$.
The right-hand side is $O(n^{d_{1}})$, which is $O(n^{d})$, and
we are done. Similar reasoning applies if for some $x\in X$, case
(iii) holds. If neither (ii) nor (iii) holds for any $x$, then
\[
\forall x\in X\,\big((\rho_{x}\le f_{1}\wedge\rho_{x}\le f_{2})\vee(\rho_{\bar{x}}\le f_{2}\wedge\rho_{\bar{x}}\le f_{1})\big),
\]
in which case, 
\[
\forall x\in X\,\big((\rho_{x}\le f_{1}\vee\rho_{\bar{x}}\le f_{1})\wedge(\rho_{x}\le f_{2}\vee\rho_{\bar{x}}\le f_{2})\big).
\]

I claim that for any function $g$, $(\forall x)(\rho_{x}\le g\vee\rho_{\bar{x}}\le g)\implies\rho<\int g$,
where $\int g$ is defined by $(\int g)(n)=1+\sum_{k<n}g(k)$. If
so $\rho$, would be bounded above by both $\int f_{1}$ and $\int f_{2}$.
If $d_{1}=d_{2}$, then $d=d_{1}+1$, and $\rho$ would be bounded
by a function in $O(n^{d_{1}+1})$, which is $O(n^{d})$. If $d_{1}\neq d_{2}$,
then suppose without loss of generality that $d_{1}<d_{2}$. Then
$\rho$ would be bounded above by $\int f_{1}$, a function in $O(n^{d_{1}+1})$,
which is again $O(n^{d})$, which completes the proof.

It remains to justify the claim. We show that $\rho(n)\le(\int g)(n)$
by induction on $n$. If $n=0$, $\rho(n)\le1\le(\int g)(n)$. Otherwise,
consider the labeled tree $W$ witnessing $\rho(n)$, and let $r$
be the label of its root. By hypothesis either $\rho_{r}$ or $\rho_{\bar{r}}$
is bounded above by $g$. Therefore, either the number of solutions
to the left of the root or the number of solutions to the right must
be bounded by $g(n-1)$. The number of solutions in the remaining
half is bounded by $\rho(n-1)$, which by induction is at most $(\int g)(n-1)$.
Therefore the total number of solutions is at most $g(n-1)+(\int g)(n-1)=(\int g)(n)$.

This concludes the proof of the first sentence. Finally, we must show
that if $T$ is $B_{d}$, $\rho(n)$ is bounded by $\chi(n,d-1)$.
The base cases $d=0$ and $d=1$ are the same as before. Otherwise,
if $(X,\mathcal{F})$ forbids $B_{d}$, then for any $x\in X$, either
$(X,\mathcal{F}_{x})$ forbids $B_{d-1}$ or $(X,\mathcal{F}_{\bar{x}})$
forbids $B_{d-1}$. Therefore, for any $x\in X$, $\rho_{x}\le\varphi(\cdot,d-1)$,
or $\rho_{\bar{x}}\le\varphi(\cdot,d-1)$. Hence, by the above claim,
$\rho\le\int\varphi(\cdot,d-1)$, which is $\varphi(\cdot,d)$.
\end{proof}
The VC density of a set system is defined to be the ``growth rate''
of the shatter function. Analogously, 
\begin{defn}
\emph{Thicket density }$\dens(X,\mathcal{F})$ is defined by
\[
\inf\{c\in\mathbb{R}:\rho(n)\in O(n^{c})\},
\]
where $\rho(n)$ is the thicket shatter function of $(X,\mathcal{F})$.
In case $\rho(n)\in O(n^{c})$ for all $c\in\mathbb{R}$, $\dens(X,\mathcal{F})=-\infty$,
and in case $\rho(n)\in O(n^{c})$ for no $c\in\mathbb{R}$, $\dens(X,\mathcal{F})=\infty$.
The \emph{dual thicket density }is defined with the dual thicket shatter
function $\rho^{\star}$ instead of $\rho$.
\end{defn}

\begin{rem}
By Theorem \ref{thm:dichotomy}, if $\dim(X,\mathcal{F})=k<\infty$,
then $\rho(n)\le\chi(n,k)$, which is $O(n^{k})$. Therefore, $\dens(X,\mathcal{F})\le\dim(X,\mathcal{F})$
for any set system $(X,\mathcal{F})$, as $\mathbb{R}\cup\{-\infty,\infty\}$-valued
quantities.
\end{rem}

\begin{lem}
\label{lem:basic density vs cardinality}The thicket density of $(X,\mathcal{F})$
is $-\infty$, $0$, or $\ge1$, depending on whether the cardinality
of $\mathcal{F}$ modulo extensionality is zero, positive and finite,
or infinite.
\end{lem}

\begin{proof}
If $\mathcal{F}$ is empty, then $\rho(n)=0$, and $\dens(X,\mathcal{F})=-\infty$.
If $\mathcal{F}$ has nonempty but has finitely many elements modulo
extensionality, say at most $B$, then for all $n$ $\rho(n)\le B$.
Therefore, $\dens(X,\mathcal{F})=0$. On the other hand, if $\mathcal{F}$
has infinitely many elements modulo extensionality, then we can extract
a sequence $(x_{i},F_{i})_{i<\omega}$ such that for any $i<j$, $F_{i}$
and $F_{j}$ agree on $x_{h}$ for $h<i$, but disagree on $x_{i}$.
(Given $\mathcal{F}$, pick some $x\in X$ such that both $\mathcal{F}_{x}$
and $\mathcal{F}_{\bar{x}}$ are nonempty. One of them, call it $\mathcal{F}^{\star}$
must be infinite; pick some $F$ come from the other, $\mathcal{F}^{\dagger}$.
Replace $\mathcal{F}$ by $\mathcal{F}^{\star}$ and repeat.) Then
the unordered tree pictured in Figure 4.1 has a solution, and, by
truncating the tree at depth $n$ for each $n$, we observe that $\rho(n)\ge n+1$,
hence $\dens(X,\mathcal{F})\ge1$.
\end{proof}
\begin{figure}
\begin{tikzpicture}
\Tree
[ . $x_0$
	[ . $x_1$
		[ . $x_2$
			[ . $\vdots$
			]
			[ . $F_2$
			]
		]
		[ . $F_1$
		]
	]
	[ . $F_0$
	]
]
\end{tikzpicture}\caption{The unordered tree of Lemma \ref{lem:basic density vs cardinality}}

\end{figure}

\subsubsection*{An alternate formulation of density}

Theorem \ref{thm:dichotomy} says roughly that, for set systems of
finite thicket dimension, trees which are balanced cannot have very
many realized leaves. Then it stands to reason that admissible trees,
i.e. those with every leaf realized, must be far from balanced. This
idea yields another way to define thicket density.
\begin{defn}
Let $\sigma(n)$ be the minimum depth of any finite $X$-labeled tree
with at least $n$ realized leaves.
\end{defn}

If $(X,\mathcal{F})$ has infinite thicket density, then $\sigma(n)=\lfloor\log_{2}(n)\rfloor$,
as witnessed by a balanced binary tree. If $(X,\mathcal{F})$ has
zero density, then $\sigma(n)$ is undefined for arbitrarily large
$n$. The interesting case is when $(X,\mathcal{F})$ has finite and
positive density, in which case $\sigma(n)$ is bounded below by $n^{\epsilon}$
for some $\epsilon>0$. In fact, the rate of growth of $\sigma$ is
the inverse of the density:
\begin{lem}
\label{lem:Fir tree dichotomy}For any set system $(X,\mathcal{F})$
of finite positive density $\delta$, if $\epsilon=\sup\{c\in\mathbb{R}:\sigma(n)\in\Omega(n^{c})\}$,
then $\epsilon=\frac{1}{\delta}$.
\end{lem}

\begin{proof}
Let $T$ be a finite $X$-labeled binary tree with $n$ realized leaves,
and let $\Delta$ be the depth of $T$. Create a new $X$-labeled
tree $T'$ by taking the balanced binary tree of depth $\Delta$,
superimposing the labeling of $T$, and labeling all remaining non-leaves
arbitrarily. Then $T'$ also has at least $n$ realized leaves, which
implies that $n\le\rho(\Delta)$. Since $T$ is arbitrary, by taking
the $T$ with the least possible $\Delta$, we obtain $n\le\rho(\sigma(n))$,
for all $n\in\mathbb{N}$. This implies that $\epsilon\ge\frac{1}{\delta}$.

On the other hand, let $T$ be a finite, balanced $X$-labeled binary
tree of depth $\Delta$ and $n$ realized leaves. Minimally prune
$T$ to obtain a tree $T'$ such that every leaf of $T'$ is realized;
then $T'$ still has $n$ realized leaves. Therefore, the depth of
$T'$ is at least $\sigma(n)$, hence $\Delta\ge\sigma(n)$. Since
this holds true of arbitrary $T$, consider $T$ of depth $\Delta$
with the greatest possible $n$. Then we have $\Delta\ge\sigma(\rho(\Delta))$,
which implies $\delta\le\frac{1}{\epsilon}$. Hence $\delta\epsilon=1$,
and we are done.
\end{proof}
\begin{figure}

\begin{tikzpicture}
\Tree
[.$1$
	[.$2$
		[.$\{1,2\}$
		]
		[.$7$
			[.$\{1,7\}$
			]
			[.$8$
				[.$\{1,8\}$
				]
				[.$\{1,9\}$
				]
			]
		]
	]
	[.$2$
		[.$3$
			[.$\{2,3\}$
			]
			[.$\{2,0\}$
			]
		]
		[.$10$
			[.$3$
				[.$\{3,10\}$
				]
				[.$\{100,10\}$
				]
			]
			[.$100$
				[.$\{3,100\}$
				]
				[.$\{3,0\}$
				]
			]
		]
	]
]
\end{tikzpicture}\caption{A $\mathbb{N}$-labeled tree with all leaves realized in ${\mathbb{N} \choose 2}$.
The depth of such trees grows like $\Omega(\sqrt{n})$, where $n$
is the size.}

\end{figure}

\section{Rank\label{sec:Ranks}}

In this section, we establish an equivalence between thicket density
and ``Shelah's local $\omega$-rank'' $R(p,\{\varphi\},\omega)$,
the local analogue of Morley rank.\footnote{We do not know of a commonly accepted name for this rank. However,
it is equivalent to $R(p,\{\varphi\},\infty)$, and the Cantor-Bendixson
rank of the space of $\varphi$-types.} From a model-theoretic perspective, the significance of this result
is probably reversed, i.e., we identify $R(p,\{\varphi\},\omega)$
with the thicket density of a particular set system.

Concretely, given any partitioned first-order formula $\varphi(x,y)$,
and for any model $M$ of a complete theory $T$, define a set system
$(M^{x},M^{y},\in)$ by $a\in b\iff M\models\varphi(a,b)$. Then,
roughly speaking, we can calculate the rank $R(\x=\x,\{\varphi\},\omega)$
of a finite $\varphi$-type modulo $T$ using the thicket shatter
function of the set system $(M^{x},M^{y},\in)$. Since the thicket
shatter function depends only on $T$ and not the particular choice
of $M$, we can carry out this calculation in \emph{any }model $M\models T$,
not only a sufficiently saturated one. 

In this section, fix a set system $(X,\mathcal{F})$, and a sufficiently
saturated model $M$ of $\Th(X,\mathcal{F},\in)$ (so that we can
calculate $\omega$-rank). Let $M_{X}$ and $M_{\mathcal{F}}$ be
the two sorts of $M$. Let $\varphi$ be the formula $\x\in\F$. (We
use $\mathtt{typewriter}$ script to distinguish syntactic variables,
e.g., $\x$, from values, e.g., $x\in X$.) By a \emph{$\varphi$-formula},
we mean a formula of the form $\varphi(a,\F)$ or $\neg\varphi(a,\F)$
for some $a\in M_{X}$. Whenever we assert that some sentence holds,
we always mean relative to the model $M$.

By a \emph{finite $\varphi$-type}, we mean a conjunction of $\varphi$-formulas,
including the \emph{empty }conjunction $\top$. Two finite $\varphi$-types
$p$ and $q$ are \emph{contradictory }in case there exists some $a\in M_{X}$
such that one of $\{\varphi(a,\F),\neg\varphi(a,\F)\}$ occurs as
a conjunct in $p$, and the other occurs as a conjunct in $q$. In
this case, we say that $p$ and $q$ \emph{disagree on }$\varphi(a,\F)$.
By $p(\mathcal{F})$ we mean the subfamily of $\mathcal{F}$ satisfying
the type $p$.
\begin{defn}
\label{def:Adm^T_=00005CPsi}For any unordered tree $T$, let $L$
be its set of leaves and $N$ be its set of non-leaves. Let $\mathcal{L}_{T}$
be the signature $\{\in\}$ expanded by a fresh set of constant symbols
$\{a_{u}:u\in N\}$ of type $X$, and $\{b_{v}:v\in L\}$ of type
$\mathcal{F}$. Given in addition a finite $\varphi$-type $p$, define
a first-order $\mathcal{L}_{T}$-theory $Adm_{p}^{T}$ by:
\begin{enumerate}
\item $p(b_{v})$ for any $v\in L$,
\item $\varphi(a_{u},b_{v})\not\leftrightarrow\varphi(a_{u},b_{w})$ if
$v,w\in L^{2}$, $u\in N$, and $v\perp_{u}w$, and
\item $\varphi(a_{u},b_{v})\leftrightarrow\varphi(a_{u},b_{w})$ if $v,w\in L^{2}$,
$u\in N$, and $v\sim_{u}w$.
\end{enumerate}
\end{defn}

\begin{rem}
\label{rem:working in M}The following hold of $Adm_{p}^{T}$:
\begin{itemize}
\item If $(X,p(\mathcal{F}))$ admits $T$, then $\Th(X,\mathcal{F})$ is
consistent with $Adm_{p}^{T}$. If $T$ is finite, then the converse
holds as well.
\item Since $M$ is sufficiently saturated, the consistency of $\Th(X,\mathcal{F})\cup Adm_{p}^{T}$
is equivalent the admissibility of $T$ in $(M_{X},p(M_{\mathcal{F}}))$.
(Concretely: the existence of injections $u\mapsto a_{u}:N\to M_{X}$
and $v\mapsto b_{v}:L\to M_{\mathcal{F}}$ such that properties (1)-(3)
of Definition \ref{def:Adm^T_=00005CPsi} hold in $M$.)
\end{itemize}
\end{rem}

\begin{lem}
Suppose that $T$ is a finite-dimensional unordered tree and that
there is an embedding of the unordered tree $S$ into $T$.\footnote{cf. Definition \ref{def:embedding}}
Then if $\Th(X,\mathcal{F})$ is consistent with $Adm_{p}^{T}$, it
is consistent with $Adm_{p}^{S}$.
\end{lem}

\begin{proof}
First observe that for every vertex $v\in T$, the subtree rooted
at $v$ contains some leaf. For otherwise (since each non-leaf has
two children) $T$ would have a complete infinite binary subtree,
and not be finite-dimensional. Fix an embedding $\iota:S\to T$, and
let $\jmath$ map leaves of $S$ to leaves of $T$, such that for
each leaf $\ell\in S$, $\jmath(\ell)$ is contained in the subtree
rooted at $\iota(\ell)$. Then notice that for any non-leaf $s$ and
leaf $\ell$ in $S$, $s\prec\ell$ in $S$ iff $\iota(s)\prec\jmath(\ell)$
in $T$. Moreover, if $\ell'$ is another leaf of $S$, then $\ell\perp_{s}\ell'$
in $S$ iff $\jmath(\ell)\perp_{\iota(s)}\jmath(\ell')$ in $T$.
Similarly, $\ell\sim_{s}\ell'$ in $S$ iff $\jmath(\ell)\sim_{\iota(s)}\jmath(\ell')$
in $T$. 

Suppose that $T$ is admissible in $(M_{X},p(M_{\mathcal{F}}))$,
witnessed by the labeling $u\mapsto a_{u}$, $v\mapsto b_{v}$. For
any non-leaf $s\in S$, label it by $a_{\iota(s)}$, and for any leaf
$t\in S$, label it by $b_{\jmath(t)}$. Since the relations $\perp$
and $\sim$ are preserved by the map $\iota$ on non-leaves and $\jmath$
on leaves, properties (2) and (3) of Definition \ref{def:Adm^T_=00005CPsi}
carry over from $T$ to $S$. Since $\jmath$ maps leaves to leaves,
property (1) of Definition \ref{def:Adm^T_=00005CPsi} carries over
from $T$ to $S$. Hence, $S$ is admissible in in $(M_{X},p(M_{\mathcal{F}}))$.
\end{proof}
Next, we define the rank $R(p,\{\varphi\},\omega)$, for finite $\varphi$-types
$p$. Here we follow the definition and notation of Pillay \cite{Pil83},
who writes $R_{\aleph_{0}}^{\varphi}$, but we omit the cardinality
$\aleph_{0}$, and just write $R^{\varphi}$.
\begin{defn}
\label{def:Rank R}For any finite $\varphi$-type $p$ in the free
variable $F$, let
\begin{itemize}
\item $R^{\varphi}(p)\ge0$ if $p$ is consistent, i.e, there exists some
$b\in M_{\mathcal{F}}$ such that $p(b)$.
\item $R^{\varphi}(p)\ge\alpha+1$ in case there is a pairwise contradictory
family of finite $\varphi$-types $\{p_{i}:i<\omega\}$ for each $i<\omega$,
such that for each $i$, $R^{\varphi}(p\wedge p_{i})\ge\alpha$.
\item For limit ordinal $\alpha$, $R^{\varphi}(p)\ge\alpha$ just in case
$R^{\varphi}(p)\ge\beta$ for all $\beta<\alpha$.
\end{itemize}
We say $R^{\varphi}(p)=\alpha$ in case $R^{\varphi}(p)\ge\alpha$
but $R^{\varphi}(p)\not\ge\alpha+1$, $R^{\varphi}(p)=-\infty$ if
$R^{\varphi}(p)\not\ge0$, and $R^{\varphi}(p)=\infty$ in case $R^{\varphi}(p)\ge\alpha$
for all $\alpha$.
\end{defn}

\begin{rem}
If $R^{\varphi}(p)\ge\omega$, then $R^{\varphi}(p)=\infty$. (Pillay
\cite{Pil83} Exercise 6.53) Hence $R^{\varphi}$ takes values in
$\omega\cup\{-\infty,\infty\}$.
\end{rem}

In fact, we can rearrange the quantifiers in the second clause to
be a little bit stronger, using some simple combinatorics:
\begin{lem}
\label{lem:extracting sequences}Suppose that $\{p_{i}:i<\omega\}$
is a sequence of pairwise contradictory finite $\varphi$-types. Then
there is an infinite set $S\subseteq\omega$ such that for any $r\in S$,
there exists $a\in M_{X}$, such that for any $s>r$ in $S$, $p_{r}$
and $p_{s}$ disagree on $\varphi(a,F)$.
\end{lem}

\begin{proof}
For any infinite set $S\subseteq\omega$, let $m$ be its least element.
Since $p_{m}$ is finite, there must be some $a\in M_{X}$ and an
infinite subset $S'\subseteq S\setminus\{m\}$ such that $p_{m}$
disagrees with $p_{n}$ on $\varphi(a,F)$ for any $n\in S'$. Let
$S_{0}=\omega$, and for each $i<\omega$, obtain $S_{i+1}$ from
$S_{i}$ in this manner. Obtain $m_{i}$ and $a_{i}$ from $S_{i}$
as above. Then $S_{0}\supset S_{1}\supset S_{2}\supset\dots$, $m_{0}<m_{1}<m_{2}<\dots$,
and $m_{i}\in S_{j}\iff i\ge j$. Moreover, for every $i<j$ $p_{m_{i}}$
and $p_{m_{j}}$ disagree on $\varphi(a_{i},F)$. Therefore, we may
take $S=\{m_{0},m_{1},m_{2},\dots\}$.
\end{proof}
\begin{defn}
[$k$-branching trees]

The 0-branching tree $T_{0}$ is a single leaf. For $k\ge1$, the
$k$-branching tree $T_{k}$ is the unordered binary tree with subtrees
$T_{k}$ and $T_{k-1}$.
\end{defn}

\begin{figure}
\begin{tikzpicture} 
\Tree
[ 
	[ 
		[ .$\circ$
		]
		[
			[ .$\circ$
			]
			[ 
				[ .$\circ$
				]
				[ .$\vdots$
				]
			]
		]
	] \edge[thick];
	[ 
		[ 
			[ .$\circ$
			]
			[
				[ .$\circ$
				]
				[ 
					[ .$\circ$
					]
					[ .$\vdots$
					]
				]
			]
		] \edge[thick];
		[
			[ 
				[ .$\circ$
				]
				[
					[ .$\circ$
					]
					[ 
						[ .$\circ$
						]
						[ .$\vdots$
						]
					]
				]
			] \edge[thick];
			[ 
				[ 
					[ .$\circ$
					]
					[
						[ .$\circ$
						]
						[ 
							[ .$\circ$
							]
							[ .$\vdots$
							]
						]
					] 
				] \edge[thick];		
				[ .$\vdots$
				]
			]
		]
	]
]

\end{tikzpicture}

\caption{The infinite 2-branching tree. The \emph{spine }is indicated by the
thick edge. The vertices of any $T_{k+1}$ can be partitioned into
the vertices on the spine, plus countably many copies of $T_{k}$.}

\end{figure}

\begin{rem}
$T_{k}$ has dimension $k$, and the number of vertices in $T_{k}$
at depth $n$ is $\chi(n,k)$.
\end{rem}

\begin{thm}
\label{thm:rank d iff d-tree}$R^{\varphi}(p)\ge k$ if and only if
$\Th(X,\mathcal{F})\cup Adm_{p}^{T_{k}}$ is consistent.
\end{thm}

\begin{proof}
We work in $M$, by Remark \ref{rem:working in M}, and induct on
$k$. If $k=0$, then $R^{\varphi}(p)\ge k$ iff there exists some
\textbf{$b\in M_{\mathcal{F}}$ }such that $p(b)$, iff the tree $T_{0}$
is admissible in $(M_{X},p(M_{\mathcal{F}}))$, iff $\Th(X,\mathcal{F})\cup Adm_{p}^{T_{0}}$
is consistent.

Otherwise, suppose that $R^{\varphi}(p)\ge k+1$, and let $\{p_{i}\}_{i<\omega}$
witness this. By Lemma \ref{lem:extracting sequences}, there is an
infinite set $S\subseteq\omega$ such that for any $r\in S$, there
is some $a\in M_{X}$, such that for any $s>r$ in $S$, $p_{r}$
and $p_{s}$ disagree on $\varphi(a,\F)$. Let $s(i)$ be the $i$-th
element of $S$, and $a^{(i)}$ witness the $\varphi$-formula distinguishing
$p_{s(i)}$ from $p_{s(j)}$, for $j>i$.

Let $N$ and $L$ be the sets of non-leaves and leaves, respectively,
of $T_{k+1}$. The tree $T_{k+1}$ can be partitioned into a single
infinite spine plus countably many copies of $T_{k}$. Let $N=\bigcup_{i<\omega}N_{i}\cup N_{s}$,
where $N_{i}$ is the set of non-leaves of the $i$-th copy of $T_{k}$,
and $N_{s}$ are the vertices down the spine. Let $L=\bigcup_{i<\omega}L_{i}$,
where $L_{i}$ is the set of leaves of the $i$-th copy of $T_{k}$.

$\F$or each $i<\omega$, we have $R^{\varphi}(p\wedge p_{s(i)})\ge k$.
Therefore, by induction, there are injections $u\mapsto a_{u}:N_{i}\to M_{X}$
and $v\mapsto b_{v}:L_{i}\to M_{\mathcal{\F}}$ satisfying properties
(1), (2), and (3). By taking the union over all $i$, we have a map
$v\mapsto b_{v}:L\to M_{\mathcal{\F}}$. Similarly, we can get a map
$u\mapsto a_{u}:N\to M_{X}$ if we specify $a_{u}$ for $u\in N_{s}$.
But for $u\in N_{s}$ of distance $i$ from the root, simply let $a_{u}=a^{(i)}$.

$\F$or any $v\in L$, $b_{v}$ satisfies $p\wedge p_{s(i)}$ for
some $i$; in particular it satisfies $p$. Thus we have verified
(1), and it remains to verify (2) and (3). $\F$ix two leaves $v,w\in L^{2}$
and a common ancestor $u\in N$.
\begin{itemize}
\item If for some $i$, $v,w\in L_{i}$ and $u\in N_{i}$, then the relations
$v\perp_{u}w$, $v\sim_{u}w$ are identical in the ambient tree $T_{k+1}$
and the $i$-th copy of $T_{k}$, hence (2) and (3) are inherited
from the given maps $N_{i}\to M_{X}$ and $L_{i}\to M_{\mathcal{\F}}$.
\item If for some $i$, $v,w\in L_{i}$ and $u\notin N_{i}$, then $u$
must be the vertex on the spine of distance $i$ from the root, and
$v\sim_{u}w$. Therefore, $a_{u}=a^{(i)}$, and since $v,w$ satisfy
$p_{s(i)}$, $b_{v}$ and $b_{w}$ must agree on $\varphi(a_{u},\F)$.
\item If for some $i<j$, $v\in L_{i}$ and $w\in L_{j}$, then $u$ must
be some vertex on the spine of distance at most $i$ from the root;
moreover, $p_{s(i)}(b_{v})$ and $p_{s(j)}(b_{w})$. Therefore, $b_{v}$
and $b_{w}$ disagree on $\varphi(a^{(i)},\F)$, and agree on $\varphi(a^{(i')},\F)$
for all $0\le i'<i$. But, $a_{u}=a^{(i)}$ just in case $v\perp_{u}w$,
and $a_{u}=a^{(i')}$ for some $0\le i'<i$ just in case $v\sim_{u}w$.
This concludes the forward direction.
\end{itemize}
In the other direction, suppose that we have injections $u\mapsto a_{u}:N\to M_{X}$
and $v\mapsto b_{v}:L\to M_{\mathcal{\F}}$ satisfying (1), (2), and
(3). Let $N_{s}$, $N_{i}$, and $L_{i}$ be as above, and let $a^{(i)}$
be the sequence of vertices along the spine. For $i<\omega$, define
$p_{i}=\{\varphi^{\star}(a^{(i)},\F)\}\cup\{\neg\varphi^{\star}(a^{(i')},\F):i'<i\}$,
where $\varphi^{\star}(a^{(i)},\F)$ is either $\varphi(a^{(i)},\F)$
or $\neg\varphi(a^{(i)},\F)$, depending on which one the leaves in
$L_{i}$ satisfy. 

For each $i$, by restricting the maps $u\mapsto a_{u}$ and $v\mapsto b_{v}$
to $N_{i}$ and $L_{i}$ respectively, we get injections which inherit
(2) and (3) from the original function. Towards (1), notice that for
each $v\in L_{i}$, $b_{v}$ satisfies $p\wedge p_{i}$, by definition
of $p_{i}$. Hence, by induction, $R^{\varphi}(p\wedge p_{i})\ge k$.
Since for each $i<j$, $p_{i}$ and $p_{j}$ disagree on $\varphi(a^{(i)},\F)$,
we have that $R^{\varphi}(p)\ge k+1$, which concludes the proof.
\end{proof}
\begin{figure}
\begin{tikzpicture} 
\Tree
[ .$a_0$
	[ .$a_{01}$
		[ .$b_{01}$
		]
		[ .$a_{02}$
			[ .$b_{02}$
			]
			[ .$a_{03}$ 
				[ .$b_{03}$
				]
				[ .$\vdots$
				]
			]
		]
	]
	[ .$a_1$
		[ .$a_{10}$
			[ .$b_{10}$
			]
			[ .$a_{11}$
				[ .$b_{11}$
				]
				[ .$a_{12}$ 
					[ .$b_{12}$
					]
					[ .$\vdots$
					]
				]
			]
		]
		[ .$a_2$
			[ .$a_{20}$
				[ .$b_{20}$
				]
				[ .$a_{21}$
					[ .$b_{21}$
					]
					[ .$a_{22}$. 
						[ .$b_{22}$
						]
						[ .$\vdots$
						]
					]
				]
			] 
			[ .$a_3$
				[ .$a_{30}$.
					[ .$b_{30}$
					]
					[ .$a_{31}$
						[ .$b_{31}$
						]
						[ .$a_{32}$ 
							[ .$b_{32}$
							]
							[ .$\vdots$
							]
						]
					] 
				]		
				[ .$\vdots$
				]
			]
		]
	]
]

\end{tikzpicture}

\caption{By Theorem \ref{thm:rank d iff d-tree}, $R^{\varphi}(p)\ge2$ is
equivalent to the admissibility of the unordered tree $T_{2}$ in
$M$, such that every $b_{ij}$ satisfies $p$.}
\end{figure}

\begin{thm}
\label{thm:density-equals-rank}The thicket density of $(X,p(\mathcal{F}))$
is identical to $R^{\varphi}(p)$, as $\mathbb{R}_{\ge0}\cup\{-\infty,\infty\}$-valued
quantities.
\end{thm}

\begin{proof}
We show that both quantities are bounded above by each other.

$\F$or some $0\le k<\omega$, suppose that $R^{\varphi}(p)\ge k$.
By Theorem \ref{thm:rank d iff d-tree}, $\Th(X,\mathcal{\mathcal{F}})$
is consistent with $Adm_{p}^{T_{k}}$. By Remark \ref{rem:working in M},
for any finite subtree $S$ of $T_{k}$, $Adm_{p}^{S}$ is consistent,
so $(X,p(\mathcal{F}))$ admits $S$. But if we obtain $S$ by truncating
$T_{k}$ to depth $n$, then $S$ has $O(n^{k})$ leaves, thus ensuring
$\dens(X,\mathcal{F})\ge k$. Hence $R^{\varphi}(p)\le\dens(X,p(\mathcal{F}))$.

Conversely, suppose that for some $0\le k<\omega$, $R^{\varphi}(p)\not\ge k$.
By Theorem \ref{thm:rank d iff d-tree}, $Adm_{p}^{T_{k}}$ is inconsistent
with $\Th(X,\mathcal{\mathcal{F}})$. By compactness, some finite
fragment of $Adm_{p}^{T_{k}}$ is inconsistent with $\Th(X,\mathcal{\mathcal{F}})$,
and hence there exists a finite subtree $S$ of $T_{k}$ such that
$Adm_{p}^{S}$ is inconsistent with $\Th(X,\mathcal{\mathcal{F}})$.
By Remark \ref{rem:working in M}, $(X,p(\mathcal{\mathcal{F}}))$
forbids $S$. Since $T_{k}$ has dimension $k$, $S$ has dimension
at most $k$. By Theorem \ref{thm:fundamental thm forbidden trees},
$\dens(X,p(\mathcal{\mathcal{F}}))\le k-1$. Hence $\dens(X,p(\mathcal{\mathcal{F}}))\le R^{\varphi}(p)$.
\end{proof}
As an immediate corollary, we deduce that thicket density is integer-valued,
in contrast to VC density. This fact was first proven by James Freitag
and Dhruv Mubayi, using elementary combinatorics, and a very short
elementary proof has been found by Ross Berkowitz. (Both of these
results are unpublished and were communicated to us in person.)

\section{Degree\label{sec:Degree}}

Each model-theoretic rank has an associated notion of \emph{degree}
(sometimes called \emph{multiplicity}); roughly, given a formula $p$,
this is the maximum number of pairwise contradictory extensions of
$p$ which each have rank no less than $p$. For any fixed rank, this
degree must be absolutely bounded by some cardinal $\kappa$, where
the rank of a formula is at least $\alpha+1$ if there are at least
$\kappa$ many pairwise contradictory extensions of rank $\alpha$.
For example, we can define the degree $D$ associated with the rank
$R^{\varphi}$ as follows.
\begin{defn}
For any finite $\varphi$-type $p$, let $D(p)$ be the greatest integer
$n<\omega$ such that there exist pairwise contradictory finite $\varphi$-types
$p_{1},\dots,p_{n}$, such that for each $1\le i\le n$, $R^{\varphi}(p\wedge p_{i})=R^{\varphi}(p)$.\footnote{If there were arbitrarily large such $n$, the there must exist an
infinite family $p_{i}$ of pairwise contradictory $\varphi$-types,
such that for each $i$, $R(p\wedge p_{i})=R(p)$, contradicting Definition
\ref{def:Rank R}.}
\end{defn}

In this section, we develop a notion of degree that seems natural
and appropriate for thicket density. Like in Section \ref{sec:Ranks},
we fix a set system $(X,\mathcal{\mathcal{F}})$ and a sufficiently
saturated model $M$ of $\Th(X,\mathcal{\mathcal{F}})$. For a finite
$\varphi$-type $p$ to have \emph{parameters from} $X$ means every
conjunct of $p$ is of the form $\varphi(x,\F)$ or $\neg\varphi(x,\F)$
for some $x\in X$. For such a type $p$, let $\dens(p)$ abbreviate
$\dens(X,p(\mathcal{\mathcal{F}}))$, where $p(\mathcal{F})=\{F\in\mathcal{F}:p(F)\}$.%

\begin{defn}
Given a vertex $v$ in an $X$-labeled tree $T$, define the finite
$\varphi$-type $p_{v}(\F)$ to be the conjunction of literals $\varphi(x_{u},\F)$
for $u\prec_{L}v$, and $\neg\varphi(x_{u},\F)$ for $u\prec_{R}v$.
Then $p_{v}$ is a finite $\varphi$-type with parameters from $X$,
and $F$ solves $v$ if and only if $p_{v}(F)$.
\end{defn}

\begin{defn}
Let $p$ be a finite $\varphi$-type with parameters from $X$. An
$X$-labeled tree $T$ \emph{factors }$p$ in case $\dens(p\wedge p_{v})=\dens(p)$,
for each leaf $v$ of $T$. Furthermore, $T$ \emph{irreducibly factors}
$p$ if, in addition, no proper extension $T'$ of $T$ factors $p$.
We say that $p$ is \emph{irreducible }if it is irreducibly factored
by a single leaf. We say that $(X,\mathcal{F})$ is irreducible in
case the empty type is.
\end{defn}

\begin{lem}
\label{lem:nondeterministic tree constr}For every finite $\varphi$-type
$p$ with parameters from $X$, if $0\le\dens(p)<\infty$, then $p$
is irreducibly factored by some finite tree.
\end{lem}

\begin{proof}
Fix the domain $X$. For each finite type $p\in S_{\varphi}(\mathcal{F})$,
we build an $X$-labeled tree $\Pi(p)$ irreducibly factoring $p$
via a nondeterministic construction.
\begin{itemize}
\item If $p$ is irreducible, $\Pi(p)$ is a single leaf.
\item Otherwise, let $x\in X$ satisfy $\dens(p\wedge\varphi(x,\F))=\dens(p\wedge\neg\varphi(x,\F))=\dens(p)$.
Let $\Pi(p)$ have root $x$ and left and right subtrees $\Pi(p\wedge\varphi(x,\F))$
and $\Pi(p\wedge\neg\varphi(x,\F))$ respectively.
\end{itemize}
We claim that this construction terminates after finitely many steps.
For otherwise, there would be infinite branch $x_{0},x_{1},x_{2},\dots$
and an infinite sequence $p=p_{0}\subset p_{1}\subset p_{2}\subset\dots$
such that for each $i<\omega$, either $p_{i+1}=p_{i}\wedge\varphi(x_{i},F)$
or $p_{i+1}=p_{i}\wedge\neg\varphi(x_{i},F)$, and $\dens(p_{i})=\dens(p_{i}\wedge\varphi(x_{i},F))=\dens(p_{i}\wedge\neg\varphi(x_{i},F))$.
Let $p_{i}^{\dagger}$ be the extension of $p_{i}$ that is not $p_{i+1}$.
Then, for each $i<\omega$ $\dens(p_{i}^{\dagger})=\dens(p)$, and
for each $i<j<\omega$, $p_{i}^{\dagger}$ and $p_{j}^{\dagger}$
agree on $\varphi(x_{h},F)$ for $h<i$, and disagree on $\varphi(x_{i},F)$.

Let $k=\dens(p)$. By Theorems \ref{thm:rank d iff d-tree} and \ref{thm:density-equals-rank},
for each $i<\omega$, we can find a consistent labeling of $T_{k}$
by elements of $M$, such that each leaf satisfies $p_{i}^{\dagger}$.
By stringing these trees along a spine labeled by the sequence $x_{0},x_{1},x_{2},\dots$,
we get a consistent labeling of $T_{k+1}$ by elements of $M$, such
that each leaf satisfies $p$. But this implies that $\dens(p)\ge k+1$,
a contradiction.
\end{proof}
\begin{figure}
\begin{tikzpicture}
\Tree
[.$x$
	[.$\circ$
	]
	[.$z$
		[.$\circ$
		]
		[.$\circ$
		]
	]
]

\end{tikzpicture}$\ \ $ \begin{tikzpicture}
\Tree
[.$\dens(\mathcal{F})$
	[.$\dens(\mathcal{F}_x)$
	]
	[.$\dens(\mathcal{F}_{\bar{x}})$
		[.$\dens(\mathcal{F}_{\bar{x}z})$
		]
		[.$\dens(\mathcal{F}_{\bar{x}\bar{z}})$
		]
	]
]

\end{tikzpicture}

\caption{An illustration suggesting $\Pi(\top)$ (on the left), and the densities
of the corresponding set systems (on the right). Each of these densities
is the same, but the set systems labeling the leaves are irreducible,
and cannot be ``split'' into subsystems of the same density by any
element of $X$.}

\end{figure}

We now show that even though there may not be a unique tree that irreducibly
factors a set system, the number of vertices in any such tree is the
same. Even stronger, the partition induced by such a tree is unique
up to rearrangement by pieces of strictly smaller density. This is
analogous to the situation for, e.g., Morley rank.
\begin{lem}
Let $p$ be a finite type with parameters from $X$. Suppose $\delta=\dens(p)$
is finite and nonnegative, and suppose $T_{1}$ and $T_{2}$ are $X$-labeled
trees that each irreducibly factor $p$. Then $T_{1}$ and $T_{2}$
have the same number of leaves. A fortiori, there is a bijection between
the leaves of $T_{1}$ and the leaves of $T_{2}$, such that $\dens(p\wedge(p_{v}\Delta p_{w}))<\delta$
for any bijective pair $(v,w)$.
\end{lem}

\begin{proof}
For $i\in\{1,2\}$, let $L_{i}$ be the set of leaves of $T_{i}$.
Since $T_{1}$ and $T_{2}$ each factor $p$, $\dens(p\wedge p_{v})=\dens(p\wedge p_{w})=\delta$
for each $v\in L_{1}$ and $w\in L_{2}$.

For any fixed $v\in L_{1}$, $\dens(p\wedge p_{v})=\max\{\dens(p\wedge p_{v}\wedge p_{w}):w\in L_{2}\}$.\footnote{In general, $\dens(X,\mathcal{F\cup G})=\max\{\dens(X,\mathcal{F}),\dens(X,\mathcal{G})\}$.
The thicket shatter function on the left is bounded below by both
thicket shatter functions on the right, and bounded above by their
sum.} Hence, for some $w\in L_{2}$, $\dens(p\wedge p_{v}\wedge p_{w})=\delta$.
However, such a $w$ must be unique; if $\dens(p\wedge p_{v}\wedge p_{w})=\dens(p\wedge p_{v}\wedge p_{w'})=\delta$,
then $T_{1}$ does not irreducibly factor $p$: the leaf $v$ can
be replaced by a non-leaf labeled by the least common ancestor of
$w$ and $w'$ in $T_{2}$.

Reasoning symmetrically, for any $w\in L_{2}$, there is a unique
$v\in L_{1}$ such that $\dens(p\wedge p_{v}\wedge p_{w})=\delta$.
Hence, the relation $R(v,w)\iff\dens(p\wedge p_{v}\wedge p_{w})=\delta$
is the graph of a bijection, so $|L_{1}|=|L_{2}|$. Furthermore, for
any bijective pair $(v,w)$, 
\[
\dens(p\wedge(p_{v}\Delta p_{w}))=\max\{\dens(p\wedge p_{v}\wedge p_{w'}),\dens(p\wedge p_{v'}\wedge p_{w}):v'\neq v,w'\neq w\},
\]
and this quantity must be strictly less than $\delta$, since we maximize
over finitely many densities all strictly less than $\delta$. 
\end{proof}
Hence, we can now define
\begin{defn}
For any finite $\varphi$-type $p$ with parameters from $X$, the
\emph{thicket degree }$\deg(p)$ is the number of leaves of any tree
irreducibly factoring $p$. The degree $\deg(X,\mathcal{F})$ is defined
to be $\deg(\top)$, the degree of the empty conjunction.
\end{defn}

\begin{rem}
Given any tree $T$ irreducibly factoring $p$, the types $p_{v}$
are pairwise contradictory as $v$ varies over the leaves of $T$,
and moreover $R^{\varphi}(p\wedge p_{v})=R^{\varphi}(p)$ for any
leaf $v$. Hence $\deg(p)\le D(p)$.
\end{rem}

\begin{rem}
By a finite $\varphi^{\star}$-type $q$ with parameters from $\mathcal{F}$,
we mean a finite conjunction of formulas of the form $\varphi(\x,F)$
and $\neg\varphi(\x,F)$, where $F$ ranges over $\mathcal{F}$. We
can define the \emph{dual degree}, $\deg^{\star}(q)$, by switching
the roles of $X$ and $\mathcal{F}$ throughout. Concretely, the dual
degree of $q$ is the number of leaves of any $\mathcal{F}$-labeled
binary tree $T$, such that $\dens^{\star}(q\wedge q_{v})=\dens^{\star}(q)$
for any leaf $v$ of $T$, and $T$ is maximal with respect to this
property. (The dual degree of $q$ is simply the degree of $q$ in
the dual set system.)
\end{rem}

\begin{rem}
Even though thicket density is equivalent to the rank $R$, the thicket
degree can differ from the degree $D$ associated with the rank $R$.
It suffices, and is easier, to exhibit a difference between the corresponding
dual quantities. Let $X$ be any infinite set, and let $\mathcal{F}$
be a partition of $X$ with arbitrarily large finite sets, but no
infinite set. Then $\dens^{\star}(X,\mathcal{F})\ge1$, since $X$
is infinite, but for any $F\in\mathcal{F}$, $\dens^{\star}(F,\mathcal{F}\restriction F)=0$,
as $F$ is finite (cf. Lemma \ref{lem:basic density vs cardinality}).
Hence, $(X,\mathcal{F})$ is (dually) irreducible: any $\mathcal{F}$-labeled
tree that irreducibly factors $(X,\mathcal{F})$ must be a single
leaf, and hence $\deg(X,\mathcal{F})=1$. On the other hand, using
compactness, we can find an element $b\in M_{\mathcal{F}}$ that has
infinitely many members and infinitely many non-members among elements
of $M_{X}$. Therefore, $R^{\star}(\varphi(\x,b))$ and $R^{\star}(\neg\varphi(\x,b))$
are both at least $1$, so $D^{\star}(\top)\ge2$.
\end{rem}

\section{Conclusion and future work}

We have established a relationship between a measure of asymptotic
growth of certain finite objects with the ordinal rank of an infinite
object, a common phenomenon in the tradition relating infinitary and
finitary combinatorics. There are several lines of inquiry that our
work raises. For example, from a technical standpoint, we do not know
what other information about a set system is contained in its thicket
shatter function. We suspect that other invariants (like the \emph{leading
coefficient}) might encode something interesting. Furthermore, there
are many identities that the rank $R^{\varphi}$ is known to satisfy,
and it would be interesting to see if they might admit a purely combinatorial
proof.

The similarity between the Sauer-Shelah Lemma and the present ``thicket''
version naturally raises the question of whether they both share a
general setting. Chase and Freitag \cite{CF} answer this question
positively by formulating a shatter function for \emph{op-rank }of
Guingona and Hill \cite{GH13}, which interpolates Shelah's 2-rank
with VC dimension, and establishing a dichotomy which interpolates
both versions of the Sauer-Shelah Lemma. Most questions about the
corresponding notion of density remain open, for example, what values
it takes ``between'' the VC and thicket cases.

Finally, as we mentioned in the Introduction, this work was originally
inspired by reading Tiuryn's paper \cite{Tiu89} separating deterministic
from nondeterministic dynamic logic. Several branches of computer
science, such as dynamic logic, descriptive complexity theory, and
program schematology, are concerned with \emph{programs that operate
over first-order structures}. A basic construction in programming
language theory is the \emph{unwind }of a program; roughly, transforming
a finitary program with recursive calls into an infinitary program
without. A fundamental invariant of the unwind of a program its its
underlying \emph{decision tree}. Many algorithmic lower bounds obtained
from lower bounds of decision trees, for example, the $\Omega(n\log n)$
lower bound on deterministic comparison sorting.

We can leverage the thicket shatter function, if it is bounded by
a polynomial, to prove stronger lower bounds that what we would otherwise
be able to show using binary decision trees. (In fact, Tiuryn does
just this.) We suspect that there might be more applications of our
results to computer science, in particular, in establishing lower
bounds for programs which branch using a model-theoretically tame
family of conditional tests.

\bibliographystyle{plain}
\bibliography{treedim}

\end{document}